\documentclass[reqno,a4paper]{amsart}
\usepackage{amssymb,setspace}
\usepackage{ifpdf}
\ifpdf
 \usepackage[hyperindex]{hyperref}%,pagebackref
\else
 \expandafter\ifx\csname dvipdfm\endcsname\relax
 \usepackage[hypertex,hyperindex]{hyperref}%,pagebackref
 \else
 \usepackage[dvipdfm,hyperindex]{hyperref}%,pagebackref
 \fi
\fi
\theoremstyle{plain}
\newtheorem{thm}{Theorem}[section]
\newtheorem{lem}{Lemma}[section]
\theoremstyle{remark}
\newtheorem{rem}{Remark}[section]
\numberwithin{equation}{section}
\DeclareMathOperator{\td}{d\mspace{-2mu}}
\allowdisplaybreaks[4]

\begin{document}

\title[The best bounds for Toader mean]
{The best bounds for Toader mean in terms of the centroidal and arithmetic means}

\author[Y.Hua]{Yun Hua}
\address[Hua]{Department of Information Engineering, Weihai Vocational College, Weihai City, Shandong Province, 264210, China}
\email{\href{mailto: Yun Hua <xxgcxhy@163.com>}{xxgcxhy@163.com}}

\author[F. Qi]{Feng Qi}
\address[Qi]{School of Mathematics and Informatics, Henan Polytechnic University, Jiaozuo City, Henan Province, 454010, China; Department of Mathematics, School of Science, Tianjin Polytechnic University, Tianjin City, 300387, China}
\email{\href{mailto: F. Qi <qifeng618@gmail.com>}{qifeng618@gmail.com}, \href{mailto: F. Qi <qifeng618@hotmail.com>}{qifeng618@hotmail.com}, \href{mailto: F. Qi <qifeng618@qq.com>}{qifeng618@qq.com}}
\urladdr{\url{http://qifeng618.wordpress.com}}

\begin{abstract}
In the paper, the authors discover the best constants $\alpha_{1}$, $\alpha_{2}$, $\beta_{1}$, and $\beta_{2}$ for the double inequalities
$$
\alpha_{1}\overline{C}(a,b)+(1-\alpha_{1}) A(a,b)< T(a,b) <\beta_{1} \overline{C}(a,b)+(1-\beta_{1})A(a,b)
$$
and
$$
\frac{\alpha_{2}}{A(a,b)}+\frac{1-\alpha_{2}}{\overline{C}(a,b)}<\frac1{T(a,b)} <\frac{\beta_{2}}{A(a,b)}+\frac{1-\beta_{2}}{\overline{C}(a,b)}
$$
to be valid for all $a,b>0$ with $a\ne b$, where
$$
\overline{C}(a,b)=\frac{2(a^{2}+ab+b^{2})}{3(a+b)},\quad A(a,b)=\frac{a+b}2,
$$
and
$$
T(a,b)=\frac{2}{\pi}\int_{0}^{{\pi}/{2}}\sqrt{a^2{\cos^2{\theta}}+b^2{\sin^2{\theta}}}\,\td\theta
$$
are respectively the centroidal, arithmetic, and Toader means of two positive numbers $a$ and $b$. As an application of the above inequalities, the authors also find some new bounds for the complete elliptic integral of the second kind.
\end{abstract}

\subjclass[2000]{Primary 26E60, 33E05; Secondary 26D15}

\keywords{Toader mean; complete elliptic integrals; arithmetic mean; centroidal mean}

\thanks{This work was supported by the Project of Shandong Province Higher Educational Science and Technology Program under grant no. J11LA57}

\thanks{This paper was typeset using \AmS-\LaTeX}

\maketitle

%%--------------------------------------------------------------------
\section{Introduction}

In~\cite{Hua-Qi-toader31}, Toader introduced a mean
\begin{align*}
T(a,b)&=\dfrac{2}{\pi}\int_{0}^{\pi/2}\sqrt{a^2\cos^2\theta+b^2\sin^2\theta}\,\td\theta\\
&=
\begin{cases}
\dfrac{2a}\pi\mathcal{E}\Biggl(\sqrt{1-\biggl(\dfrac{b}a\biggr)^2}\,\Biggr),&a>b,\\
\dfrac{2b}\pi\mathcal{E}\Biggl(\sqrt{1-\biggl(\dfrac{b}a\biggr)^2}\,\Biggr),&a<b,\\
a,&a=b.
\end{cases}
\end{align*}
where
$$
{\mathcal{E}}={\mathcal{E}}(r)=\int_{0}^{\pi/2}\sqrt{1-r^2\sin^2\theta}\,\td\theta
$$
for $r\in[0,1]$ is the complete elliptic integral of the second kind.
The quantities
\begin{align*}
\overline{C}(a,b)&=\frac{2(a^2+ab+b^2)}{3(a+b)}, &A(a,b)&=\frac{a+b}{2}, & S(a,b)&=\sqrt{\frac{a^2+b^2}{2}}\,
\end{align*}
are called in the literature the centroidal, arithmetic, and quadratic means of two positive real numbers $a$ and $b$ with $a\ne b$. For $p\in \mathbb{R}$ and $a,b>0$ with $a\ne b$, the $p$-th power mean $M_{p}(a,b)$ is defined by
\begin{equation}\label{eq1.3}
M_{p}(a,b)=
\begin{cases}
\biggl(\dfrac{a^{p}+a^{p}}{2}\biggr)^{{1}/{p}},&p\ne 0,\\
\sqrt{ab}\,,&p=0.
\end{cases}
\end{equation}
It is well known that
\begin{equation*}
  M_{-1}(a,b)<A(a,b)=M_1(a,b)<\overline{C}(a,b)<S(a,b)=M_2(a,b)
\end{equation*}
for all $a,b>0$ with $a\ne b$.
\par
In~\cite{Hua-Qi-toader39}, Vuorinen conjectured that
\begin{equation}\label{eq1.4}
  M_{3/2}(a,b)<T(a,b)
\end{equation}
for all $a,b>0$ with $a\ne b$. This conjecture was verified by Qiu and Shen~\cite{Hua-Qi-toader31Shen} and by Barnard, Pearce, and Richards~\cite{Hua-Qi-toader311}.
In~\cite{Hua-Qi-toader312}, Alzer and Qiu presented that
\begin{equation}\label{eq1.5}
  T(a,b)<M_{(\ln2)/\ln(\pi/2)}(a,b)
\end{equation}
for all $a,b>0$ with $a\ne b$, which gives a best possible upper bound for Toader mean in terms of the power mean.
From~\eqref{eq1.4} and~\eqref{eq1.5}, one concludes that
\begin{equation}\label{eq1.6}
  A(a,b)<T(a,b)<S(a,b)
\end{equation}
for all $a,b>0$ with $a\ne b$.
In~\cite{Hua-Qi-toader34}, the authors demonstrated that the double inequality
\begin{equation}\label{eq1.7}
\alpha S(a,b)+(1-\alpha)A(a,b)<T(a,b)<\beta S(a,b)+(1-\beta)A(a,b)
\end{equation}
holds for all $a,b>0$ with $a\ne b$ if and only if $\alpha\le \frac{1}{2}$ and $\beta\ge \frac{4-\pi}{(\sqrt{2}\,-1)\pi}$.
\par
Motivated by the double inequality~\eqref{eq1.7}, we naturally ask a question: What are the best constants $\alpha_1,\alpha_2,\beta_1,\beta_2\in(0,1)$ such that the
double inequalities
\begin{equation}\label{eq3.1}
\alpha_1 \overline{C}(a,b)+(1-\alpha_1)A(a,b)<T(a,b) <\beta_1 \overline{C}(a,b)+(1-\beta_1)A(a,b)
\end{equation}
and
\begin{equation}\label{eq3.13}
\frac{\alpha_2}{A(a,b)}+\frac{1-\alpha_2}{\overline{C}(a,b)} <\frac{1}{T(a,b)}<\frac{\beta_2}{A(a,b)}+\frac{1-\beta_2}{\overline{C}(a,b)}
\end{equation}
hold for all $a,b>0$ with $a\ne b$?
\par
The main aim of this paper is to affirmatively answer the above question.

\begin{thm}\label{th3.1}
The double inequality~\eqref{eq3.1} holds for all $a,b>0$ with $a\ne b$ if and only if $\alpha_1\le \frac{3}{4}$ and $\beta_1 \ge\frac{12}{\pi}-3$.
\end{thm}

\begin{thm}\label{th3.2}
The double inequality~\eqref{eq3.13} holds for all $a,b>0$ with $a\ne b$ if and only if $\alpha_1\le \pi-3$ and $\beta_1 \ge\frac{1}{4}$.
\end{thm}

As an immediate applications of Theorem~\ref{th3.1}, we will derive a new bounds in terms of elementary functions for the complete elliptic integral of the second kind.

\begin{thm}\label{cor4.1}
For $r\in(0,1)$ and $r'=\sqrt{1-r^2}\,$, we have
\begin{multline}\label{eq4.1}
\frac{\pi}{2}\biggl[\frac{1+r'+(r')^2}{2(1+r')}+\frac{1+r'}{8}\biggr]<\mathcal{E}(r)\\
<\frac{\pi}{2}\biggl[\biggl(\frac{8}{\pi}-2\biggr)\frac{1+r'+(r')^2}{1+r'} +\biggl(2-\frac6\pi\biggr)(1+r')\biggr].
\end{multline}
\end{thm}

In Section~\ref{compar-sec} we will compare the above main results with some well-known ones.

\begin{rem}
Some estimates for the three kinds of complete elliptic integrals were established in~\cite{Hua-Qi-toader312, Jiang-toader29, Hua-Qi-toader38, Hua-Qi-toader310, Hua-Qi-toader311, Bracken-Expo-01, MIA-1754.tex, mia-qi-cui-xu-99, qh, elliptic-mean-comparison-rev2.tex, Hua-Qi-toader39, Hua-Qi-toader313, Yin-Qi-Ellip-Int.tex}. and there is a short review and survey in~\cite[pp.~40\nobreakdash--46]{refine-jordan-kober.tex-JIA} for these estimates.
\end{rem}

\section{Lemmas}

For proving our main results, we need the following lemmas.
\par
For $0<r<1$, denote $r'=\sqrt{1-r^2}\,$. It is known that Legendre's complete elliptic integrals of the first and second kind are defined respectively by
$$
\begin{cases}\displaystyle
{\mathcal{K}}={\mathcal{K}}(r)=\int_{0}^{{\pi}/{2}}\frac1{\sqrt{1-r^2\sin^2\theta}\,}\td \theta,\\
{\mathcal{K}}'={\mathcal{K}}'(r)={\mathcal{K}}(r'),\\
{\mathcal{K}}(0)=\dfrac\pi2,\\
{\mathcal{K}}(1)=\infty
\end{cases}
$$
and
$$
\begin{cases}\displaystyle
{\mathcal{E}}={\mathcal{E}}(r)=\int_{0}^{\pi/2}\sqrt{1-r^2\sin^2\theta}\,\td\theta,\\
{\mathcal{E}}'={\mathcal{E}}'(r)={\mathcal{E}}(r'),\\
{\mathcal{E}}(0)=\dfrac\pi2,\\
{\mathcal{E}}(1)=1.
\end{cases}
$$
See~\cite{Hua-Qi-toader36, Hua-Qi-toader37}.
For $0<r<1$, the following formulas were presented in~\cite[Appendix~E, pp.~474\nobreakdash--475]{Hua-Qi-toader38}:
\begin{gather*}
 \frac{\td \mathcal{K}}{\td r}=\frac{\mathcal{E}-(r')^2{\mathcal{K}}}{r(r')^2},\quad
\frac{\td \mathcal{E}}{\td r}=\frac{\mathcal{E}-\mathcal{K}}{r},\quad
\frac{\td ({\mathcal{E}}-(r')^2{\mathcal{K}})}{\td r}=r{\mathcal{K}},\\
\frac{\td ({\mathcal{K}}-{\mathcal{E}})}{\td r}=\frac{r{\mathcal{E}}}{(r')^2},
\quad
{\mathcal{E}}\biggl(\frac{2\sqrt{r}\,}{1+r}\biggr)=\frac{2\mathcal{E}-(r')^2\mathcal{K}}{1+r}.
\end{gather*}

\begin{lem}[{\cite[Theorem~3.21]{Hua-Qi-toader38}}]\label{lem2.1}
The function $\frac{\mathcal{E}-(r')^2\mathcal{K}}{r^2}$ is strictly increasing from $(0,1)$ onto $\bigl(\frac\pi4,1\bigr)$.
\end{lem}

\begin{lem}\label{lem2.3}
The function $5\mathcal{E}(r)-3(r')^2\mathcal{K}(r)$ is positive and strictly increasing on $(0,1)$.
\end{lem}

\begin{proof}
Let $f(r)=5\mathcal{E}(r)-3(r')^2\mathcal{K}(r)$ for $r\in(0,1)$ and $r'=\sqrt{1-r^2}\,$. A simple computation leads to
\begin{equation*}
  f'(r)=\frac{2\mathcal{E}(r)-2\mathcal{K}(r)+3r^2\mathcal{K}(r)}{r}\triangleq\frac{g(r)}{r}.
\end{equation*}
A direct differentiation yields
\begin{align*}
  g'(r)=\frac{r(\mathcal{E}(r))+3(r')^2\mathcal{K}(r)}{(r')^2}>0
\end{align*}
for all $r\in(0,1)$, that is, the function $g(r)$ is strictly increasing on $(0,1)$. Hence, it is derived that $g(r)>g(0)=0$, that $f'(r)>0$, that $f(r)$ is increasing on $(0,1)$, and that $f(x)>f(0)=\pi>0$.
\end{proof}

\begin{lem}[{\cite[Theorem~1.25]{Hua-Qi-toader38}}]\label{lem2.2}
For $-\infty<a<b<\infty$, let $f,g:[a,b]\to{\mathbb{R}}$ be continuous on $[a,b]$, differentiable on $(a,b)$, and $g'(x)\ne 0$ on $(a,b)$. If $\frac{f'(x)}{g'(x)}$ is increasing \textup{(}or decreasing respectively\textup{)} on $(a,b)$, so are
$$
\frac{f(x)-f(a)}{g(x)-g(a)}\quad\text{and}\quad  \frac{f(x)-f(b)}{g(x)-g(b)}.
$$
\end{lem}

\begin{rem}
Lemma~\ref{lem2.2} and its variants have been extensively applied in, for example,~\cite{B, wilker-Pinelis, QiBerg.tex, elliptic-mean-comparison-rev2.tex} and many references listed in~\cite{refine-jordan-kober.tex-JIA}, and have been generalized in, for example, \cite[Lemma~2.2]{Koumandos-Pedersen-287} and~\cite{wilker-Pinelis} and closely related references therein. For more information, please read the first sentence after~\cite[p.~582, Lemma~2.1]{B}, the references~\cite{avv-siam, Biernacki-Krzyz-1955, Ponnusamy-Vuorinen-97, Bessel-ineq-Dgree-CM.tex} and~\cite[Remark~2.2]{QiBerg.tex}
\end{rem}

\section{Proofs of main results}

Now we are in a position to prove our main results.

\begin{proof}[Proof of Theorem~\ref{th3.1}]
Without loss of generality, we assume that $a>b$. Let $t=\frac{b}a\in(0,1)$ and $r=\frac{1-t}{1+t}$. Then
\begin{equation*}
\frac{T(a,b)-A(a,b)}{\overline{C}(a,b)-A(a,b)} =\frac{\frac{\pi}{2}\mathcal{E}'(t)-\frac{1+t}{2}}{\frac{2}{3}\frac{1+t+t^2}{1+t}-\frac{1+t}{2}}
  =\frac{\frac{2}{\pi}\mathcal{E}\bigl(\frac{2\sqrt{r}\,}{1+r}\bigr) -\frac{1}{1+r}}{\frac{1}{3}\frac{r^2}{1+r}}
  =3\frac{\frac{2}{\pi}[2\mathcal{E}-(r')^2\mathcal{K}]-1}{r^2}.
\end{equation*}
Let $f_1(r)=\frac{2}{\pi}[2\mathcal{E}-(r')^2\mathcal{K}]-1$, $f_2(r)=r^2$, and
\begin{equation*}
f(r)=3\frac{f_1(r)}{f_2(r)}=3\frac{\frac{2}{\pi}\bigl[2\mathcal{E}-(r')^2\mathcal{K}\bigr]-1}{r^2}.
\end{equation*}
Simple computations lead to
\begin{align*}
  f_1(0)&=f_2(0)=0,&
  f_1'(r)&=\frac{2}{\pi}\frac{\mathcal{E}-(r')^2\mathcal{K}}{r},&
  f_2'(r)&=2r,&
 \frac{f_1'(r)}{f_2'(r)}&=\frac{1}{\pi}\frac{\mathcal{E}-(r')^2\mathcal{K}}{r^2}.
\end{align*}
Combining this with Lemmas~\ref{lem2.1} and~\ref{lem2.2} reveals that the function $f(r)$ is strictly increasing on $(0, 1)$. Further making use of L'H\^opital's rule gives
\begin{equation*}
  \lim_{r\to 0^+}f(r)=\frac{3}{4}\quad\text{and}\quad \lim_{r\to 1^-}f(r)=\frac{12}{\pi}-3.
\end{equation*}
Theorem~\ref{th3.1} is thus proved.
\end{proof}

\begin{proof}[Proof of Theorem~\ref{th3.2}]
Without loss of generality, we assume that $a>b$. Let $t=\frac{b}a\in(0,1)$ and $r=\frac{1-t}{1+t}$. Then
\begin{multline*}
\frac{{1}/{T(a,b)}-{1}/{\overline{C}(a,b)}}{{1}/{A(a,b)} -{1}/{\overline{C}(a,b)}}
=\frac{1-{T(a,b)}/{\overline{C}(a,b)}}{T(a,b) \bigl[{1}/{A(a,b)}-{1}/{\overline{C}(a,b)}\bigr]}\\
=\frac{1-{\frac{2}{\pi}\mathcal{E}'(t)}\big/{\frac{2}{3}\frac{1+t+t^2}{1+t}}} {\frac{2}{\pi}\mathcal{E}'(t)\bigl[\frac{2}{1+t}-\frac{3(1+t)}{2(1+t+t^2)}\bigr]}
=\frac{3+r^2-\frac6\pi\bigl[2\mathcal{E}-(r')^2\mathcal{K}\bigr]}{\frac{2}{\pi} r^2 \bigl[2\mathcal{E}-(r')^2\mathcal{K}\bigr]}.
\end{multline*}
Let $f_1(r)=3+r^2-\frac6\pi\bigl[2\mathcal{E}-(r')^2\mathcal{K}\bigr]$, $f_2(r)=\frac{2}{\pi}r^2\bigl[2\mathcal{E}-(r')^2\mathcal{K}\bigr]$, and
\begin{equation*}
f(r)=\frac{f_1(r)}{f_2(r)}=\frac{3+r^2-\frac6\pi\bigl[2\mathcal{E}-(r')^2\mathcal{K}\bigr]} {\frac{2}{\pi}r^2\bigl[2\mathcal{E}-(r')^2\mathcal{K}\bigr]}.
\end{equation*}
Simple computations lead to
\begin{gather*}
  f_1(0)=f_2(0)=0,\quad
  f_1'(r)=2r-\frac6\pi\frac{\mathcal{E}-(r')^2\mathcal{K}}{r},\\
  f_2'(r)=\frac{2}{\pi}\bigl[5\mathcal{E}-3(r')^2\bigr]\mathcal{K}r,\quad
 \frac{f_1'(r)}{f_2'(r)}=\frac{2-\frac6\pi\frac{[\mathcal{E}-(r')^2\mathcal{K}]}{r^2}} {\frac{2}{\pi}[5\mathcal{E}-3(r')^2\mathcal{K}]}.
\end{gather*}
Combining this with Lemmas~\ref{lem2.1}, \ref{lem2.3}, and~\ref{lem2.2} yields that the function $f(r)$ is strictly decreasing on $(0, 1)$. Making use of L'H\^opital's rule shows that
$$
\lim_{r\to 0^+}f(r)=\frac{1}{4}\quad\text{and}\quad
  \lim_{r\to 1^-}f(r)={\pi}-3 .
$$
Thus, Theorem~\ref{th3.2} is proved.
\end{proof}

\begin{proof}[Proof of Theorem~\ref{cor4.1}]
Without loss of generality, assume that $a>b$. Substituting $r'=\frac{b}a$, $\alpha_1=\frac34$, and $\beta_1=\frac{12}\pi-3$ into Theorem~\ref{th3.1} produces Thorem~\ref{cor4.1}.
\end{proof}

\section{Comparisons with some known results}\label{compar-sec}

In~\cite{Hua-Qi-toader34}, it was obtained that
\begin{multline}\label{eq4.2}
\frac{\pi}{2}\biggl[\frac{1}{2}\sqrt{\frac{1+(r')^2}{2}}\,+\frac{1+r'}{4}\biggr]<\mathcal{E}(r)\\
<\frac{\pi}{2}\biggl[\frac{4-\pi}{\bigl(\sqrt{2}\,-1\bigr)\pi}\sqrt{\frac{1+(r')^2}{2}}\, +\frac{(\sqrt{2}\,\pi-4)(1+r')}{2\bigl(\sqrt{2}\,-1\bigr)\pi}\biggr]
\end{multline}
for all $r\in(0,1)$ and $r'=\sqrt{1-r^2}\,$.
Guo and Qi proved in~\cite{MIA-1754.tex} that
\begin{equation}\label{eq4.3}
  \frac{\pi}{2}-\frac{1}{2}\ln \frac{(1+r)^{1-r}}{(1-r)^{1+r}}<\mathcal{E}(r)< \frac{\pi-1}{2}+\frac{1-r^2}{4r}\ln \frac{1+r}{1-r}
\end{equation}
for all $r\in(0,1)$.
It was pointed out in~\cite{Hua-Qi-toader34} that the bounds in~\eqref{eq4.2} for $\mathcal{E}(r)$ are better than those in~\eqref{eq4.3} for some $r\in (0,1)$.
Very recently, Yin and Qi obtained in~\cite{Yin-Qi-Ellip-Int.tex} that
\begin{equation}\label{eq-l}
\frac{\pi}{2}\frac{\sqrt{6+2\sqrt{1-r^2}\,-3r^2}\,}{2\sqrt{2}\,}\le \mathcal{E}(r)\le \frac{\pi}{2}\frac{\sqrt{10-2\sqrt{1-r^2}\,-5r^2}\,}{2\sqrt{2}\,}.
\end{equation}
\par
Let
$$
g(x)=\frac{1+x+x^2}{2(1+x)}+\frac{1+x}{8}-\biggl(\frac{1}{2} \sqrt{\frac{1+x^2}{2}}\,+\frac{1+x}{4}\biggr)
$$
for $x\in(0,1)$. Then a simplification leads to
\begin{equation*}
  g(x)=\frac{3x^2+2x+3-2(1+x)\sqrt{2(1+x^2)}\,}{8(1+x)}.
\end{equation*}
Since
\begin{equation*}
\bigl(3x^2+2x+3\bigr)^2-\bigl[2(1+x)\sqrt{2(1+x^2)}\,\bigr]^2
=(1-x)^4>0,
\end{equation*}
the lower bound in~\eqref{eq4.1} for $\mathcal{E}(r)$ is better than the one in~\eqref{eq4.2}.
\par
Since
$$
\frac{1+x+x^2}{2(1+x)}+\frac{1+x}{8}>\frac{\sqrt{6+2x-3(1-x^2)}\,}{2\sqrt{2}\,}
$$
is equivalent to
$$
\bigl(5x^2+6x+5\bigr)^2>8(x+1)^2\bigl(3x^2+2x+3\bigr)
$$
and $(x-1)^4>0$, the lower bound in~\eqref{eq4.1} for $\mathcal{E}(r)$ is better than the one in~\eqref{eq-l}.
\par
Let
\begin{align*} J(r)&=\frac{\pi}{2}\biggl[\biggl(\frac8\pi-2\biggr)\frac{1+r'+(r')^2}{1+r'}+(2-\frac6\pi)(1+r')\biggr],\\
D(r)&=\frac{\pi}{2}\biggl[\frac{4-\pi}{\bigl(\sqrt{2}\,-1\bigr)\pi}\sqrt{\frac{1+(r')^2}{2}}\, +\frac{(\sqrt{2}\,\pi-4)(1+r')}{2\bigl(\sqrt{2}\,-1\bigr)\pi}\biggr],\\ Q(r)&=\frac{\pi-1}{2}+\frac{1-r^2}{4r}\ln \frac{1+r}{1-r},\\
Y(r)&=\frac{\pi}{2}\frac{\sqrt{10-2\sqrt{1-r^2}\,-5r^2}\,}{2\sqrt{2}\,}.
\end{align*}
The values of these functions at points $0.1$, $0.2$, $0.3$, $0.4$, $0.5$, $0.6$, $0.7$, $0.8$, $0.9$ can be numerical computed and listed in Table~\ref{JDQY-values}.
\begin{table}[htbp]
\caption{Values of $J(r)$, $D(r)$, $Q(r)$, and $Y(r)$}
\begin{tabular}{c|c|c|c|c}
$r$&$J(r)$&$D(r)$&$Q(r)$&$Y(r)$\\
\hline
$0.1$&$ 1.566862174\dotsm $&$1.566862736\dotsm $&$1.567456298\dotsm $&$1.566866887\dotsm$ \\
$0.2$&$ 1.554972309\dotsm $&$1.554981471\dotsm $&$1.557354457\dotsm $&$1.555049510\dotsm$ \\
$0.3$&$ 1.534853276\dotsm $&$1.534901499\dotsm $&$1.540234393\dotsm $&$1.535259718\dotsm$ \\
$0.4$&$ 1.506007907\dotsm $&$1.506169094\dotsm $&$1.515627704\dotsm $&$1.507368120\dotsm$ \\
$0.5$&$ 1.467637170\dotsm $&$1.468061483\dotsm $&$1.482775936\dotsm $&$1.471228040\dotsm$ \\
$0.6$&$ 1.418485626\dotsm $&$1.419455645\dotsm $&$1.440474824\dotsm $&$1.426746617\dotsm$ \\
$0.7$&$ 1.356514851\dotsm $&$1.358548915\dotsm $&$1.386741519\dotsm $&$1.374078083\dotsm$ \\
$0.8$&$ 1.278097245\dotsm $&$1.282149209\dotsm $&$1.317984092\dotsm $&$1.314222496\dotsm$ \\
$0.9$&$ 1.175305090\dotsm $&$1.183095913\dotsm $&$1.226197273\dotsm $&$1.251499407\dotsm$ \\
\hline
\end{tabular}
\label{JDQY-values}
\end{table}
This implies that the upper bound in~\eqref{eq4.1} for $\mathcal{E}(r)$ are better than those in~\eqref{eq4.2}, \eqref{eq4.3}, and~\eqref{eq-l} for some $r\in(0,1)$.
\par
In conclusion, the double inequality~\eqref{eq4.1} is better than some known results in~\cite{Hua-Qi-toader34, MIA-1754.tex, Yin-Qi-Ellip-Int.tex} somewhere.

\end{document}